\documentclass[10pt,leqno]{amsart}

\usepackage{amssymb,latexsym}
\usepackage{graphicx,epsfig,color}

\usepackage{a4wide}
\usepackage{amsthm}
\usepackage{latexsym}
\usepackage[active]{srcltx}

\newtheorem{theorem}{Theorem}[section]
\newtheorem{definition}[theorem]{Definition}

\newtheorem{lemma}[theorem]{Lemma}
\newtheorem{remark}[theorem]{Remark}
\newtheorem{proposition}[theorem]{Proposition}

\newtheorem{corollary}[theorem]{Corollary}
\newtheorem{example}[theorem]{Example}

\def\R{{\mathbb R}}
\def\C{{\mathbb C}}

\def\al{\alpha}
\def\be{\beta}
\def\DE{\Delta}

\def\de{\delta}

\def\ga{\gamma}
\def\GA{\Gamma}

\def\ve{\varepsilon}

\def\la{\lambda}

\def\om{\omega}
\def\va{\varphi}
\def\ta{\tau}

\def\Ga{\Gamma}
\def\si{\sigma}

\def\g{\mathfrak{g}}

\def\h{\mathfrak{h}}

\def\p{\mathfrak{p}}

\def\s{\mathfrak{s}}

\def\g{\mathfrak g}
\def\h{\mathfrak h}

\def\noi{\noindent}
\def \trans{\,{}^t\!}

\def\R{\mathbb{R}}
\def\C{\mathbb{C}}

\def\nn{\nonumber}

\def \sgn {\text{\rm sgn\,}}

\def\Id{{\mathbb I}}

\def\CC{{\mathcal C}}

\def\LL{{\mathcal L}}

\def\SS{{\mathcal S}}

\def\U{{\mathcal U}}
\def\V{{\mathcal V}}

\def\CC{{\mathcal C}}

\def\iy{\infty}

\def\hb#1{\hbox{#1}}
\def\val#1{\vert #1\vert}

\def\no#1{\Vert #1\Vert }
\def\noop#1{\Vert #1\Vert_{\rm op} }

\def\exp#1{\hb{exp}(#1)}

\def\res#1{\Big|}
\def\inv{^{-1}}

\def\hb #1{\hbox{#1}}

\def\hb#1{\hbox{#1}}
\def\val#1{\vert #1\vert}

\def\ti{\times}

\def\dim#1{\hb{dim}(#1)}

\def\-{(\pm)}

\def\L1#1{L^1(#1)}

\def\L#1#2{L^{#1}(#2)}
\def\l#1#2{L^{#1}(#2)}
\def\Im{\mathrm{\, Im \,}}
\def\Re{\mathrm{\, Re\, }}

\begin{document}

\title[Uniqueness of solutions to Schr\"odinger equations]
{Uniqueness of Solutions to Schr\"odinger equations on 2-step nilpotent Lie groups}
\author{Jean Ludwig, Detlef~M\"uller}
\thanks{The second author was supported by a one  month research invitation from the University
Paul Verlaine-Metz in 2010-2011} 
\begin{abstract}
Let $ \g=\g_1\oplus\g_2,[\g,\g] =\g_2, $ be a nilpotent Lie algebra of step 2, $ V_1,\cdots, V_m $ a basis of $ \g_1 $ and $ L=\sum_{j,k}^{m}a_{jk}V_j V_k $ be a  left-invariant differential operator on $ G=\exp \g $, where $ (a_{jk})_{jk}\in M_n(\R) $ is symmetric. It is shown that if a solution $ w(t,x) $ to the  Schr\"odinger equation 
$\partial_t w(t,g)=i Lw(t,g), g\in G, t\in\R, 
 w(0,g)=f(g)$, satisfies a suitable  Gaussian type estimate at time  $t= 0 $ and at some time  $ t=T\ne 0 $, then $ w=0 $. The proof is based on Hardy's  uncertainty principle, on explicit computations within   Howe's oscillator semigroup and on the methods developed in \cite{MRAnn} and \cite{MRMaAn}. Our results extend the work by Ben Sa\"{i}d and Thangavelu \cite{ST}, in which the authors study the Schr\"odinger equation associated to the sub-Laplacian on the Heisenberg group.
\end{abstract}
\maketitle

\setcounter{equation}{0}
\section{Introduction}\label{intro}
In a recent preprint \cite{ST}, Ben Sa\"{i}d and Thangavelu have extended a uniqueness property for solutions to the free Schr\"odinger equation on Euclidean space to  the Schr\"odinger equation associated to the sub-Laplacian on the Heisenberg group (see Theorem \ref{ST}).  Their proof exploits in a crucial way the symmetry properties of the sub-Laplacian and  makes use  of a detailed  analysis  of the representation theory of the Heisenberg group $ \mathbb H_n,$  which eventually allows to reduce to  a variant of Hardy's  classical uncertainty principle  \cite{Hardy} for the Hankel transform. 

In this paper we shall generalize the result of Ben Sa\"{i}d and Thangavelu to every  step two nilpotent Lie group, and at the same time show that it holds true even  for a much wider  class of  second order left-invariant differential operators $ L $ on the corresponding Lie group. These operators need not  be hypoelliptic. Our proof is based on a quite different approach,  which at the same time appears conceptually simpler, as it allows in the end for a reduction directly to Hardy's classical theorem, and avoids cumbersome computations in irreducible representations.

\subsection{The Uncertainty Principle}
Various forms of the  \textit{Uncertainty Principle} have again been in the research focus  in recent years, for instance,   among others, in the papers  \cite{BDJ}, \cite{CF1},\cite{CEKKPV}, \cite{EKPV1},\cite{EKPV2}  and \cite{EKPV3}. 

Of importance to us will be  Hardy's incertainty principle  \cite{Hardy}. 
On $ \R^{n},  $ with Fourier transform
\begin{eqnarray}\label{}
\nn \hat f(\xi)&:=&\int_{\R^{n}}f(x)e^{{-i\xi\cdot x}}dx,\quad  f\in \l1 {\R^{n}},\quad  \xi\in \R^{n},
\end{eqnarray}
this principle states the following:
\begin{theorem}\label{Hardy}
Let $ \al,\be\in \R^*_+. $ Let $ g: \R^{n}\to \C $ be a measurable function such that for every $x$ respectively $\xi$ in $\R^n$
\begin{eqnarray}\label{}
\nn \val{g(x)}\leq C e^{\frac{-\val x^{2}}{4\al}},\quad \val{\hat g(\xi)}\leq Ce^{\frac{-\val \xi^{2}}{\be}}.
\end{eqnarray}
Then the following hold true: 
if $\al\be<1  $, then $ g=0,$ and if $ \al\be=1 $, then $ g(x)=c e^{\frac{-\val {x }^{2}}{4\al}}, \ x\in\R^{n} $.
 \end{theorem}
The classical proof uses complex analysis; a  real variable proof can be found in the paper \cite{CEKKPV}.

\begin{remark}\label{}
  It is well-known \cite{EKPV1}  that Hardy's uncertainty principle is equivalent to the uniqueness for solutions to the Schr\"odinger equation
\begin{eqnarray*}
\partial _tw&=&i\DE w \quad \ \text{ on }\R^{n}\times \R,\\
 w(x,0)&=&f(x),\quad x\in\R^{n},
\end{eqnarray*}
where $ \DE=\sum_j \partial_j^{2} $ denotes the Laplacian on $ \R^{n} $.
 \end{remark}

 Indeed
we have that 
\begin{eqnarray}\label{}
\nn w(x,t)&=&(e^{it\DE }f)(x) =\int_{\R^n}\frac{e^{i\frac{\val{x-y}^2}{4t}}}
{(4\pi i t)^{n/2}}f(y)\,dy
\nn  =\frac
{e^{i\frac{\val{x}^2}{4t}}}
{(4\pi i t)^{n/2}}\int_{\R^n}e^{-i
\frac{x\cdot y}{2t}}
e^{i\frac{\val{y}^2}{4t}}f(y)\,dy.
\end{eqnarray}
Let now $ g(x):=e^{i
\frac{\val{x}^2}{4t}
}f(x), x\in\R .$
Then 
\begin{eqnarray}\label{}
\nn \val {w(x,0)}=\val {g(x)},\quad \vert w(x,t)\vert ={\frac
{1}{(4\pi  \val t)^{n/2}}}\val{\hat g(\frac{x}{2t})},\quad x\in\R^{n}, t\in\R^*.
\end{eqnarray}
Hence if for some $ a,b\in\R^*_+$ and  $T\in\R^* $
\begin{eqnarray}\label{}
\nn \vert w(x,0)\vert\leq Ce^{-\frac{\val x ^{2}}{4a}}, \quad
\val{w(x,T)}&\leq &C e^{-\frac{\val x^{2}}{4b}} ,
\end{eqnarray}
 then 
\begin{eqnarray}\label{}
\nn \val {g(x)}\leq  Ce^{-\frac{\val x ^{2}}{4a}},\quad
 \val {\hat g(\xi)}\leq C_T Ce^{-\frac{\val \xi^{2}}{b/T^{2}}}. 
\end{eqnarray}
Therefore by Hardy's Theorem, if $ ab<T^{2} $, then $ g=0 $, which implies that $ f=0 $ and then also $ w=0 $.

Replace now the Laplacian $\Delta$ on $\R^n$ by the sub-Laplacian $ \LL $ on  the Heisenberg group $ \mathbb H_n$.  In the preprint \cite{ST}, Ben Sa\"{i}d and Thangavelu prove the following uniqueness theorem for  solutions to the associated  Schr\"odinger equation:
 \begin{theorem}\label{ST}
 Let $ a>0, b>0, T\in\R$, such that $ ab<T^{2} $. 
 Let $ (h_t)_{t>0} $ be the heat kernel associated to the sub-Laplacian $ \LL $ on $ \mathbb H_n $. If for a measurable function $ f:\mathbb H_n\to\C $
 \begin{eqnarray}\label{}
\nn \val{f(x)}&\leq&Ch_a(x),\\
\nn \val{(e^{{iT \LL}}f)(x)}&\leq&Ch_b(x),\quad   x\in\mathbb H_n, 
\end{eqnarray}
then $ f=0 $.

 \end{theorem}

  \subsection{Step two nilpotent Lie groups}
Let $ \g $ be a real finite dimensional nilpotent non-abelian  Lie algebra of step 2. Then 
\bigskip
\begin{eqnarray}\label{}
 \nn \g&=&\g_1\oplus\g_2,\\
\nn \{0\}\ne [\g,\g]  &=&[\g_1,\g_1]\subset \g_2,\ [\g,\g_2] =\{0\}. 
\end{eqnarray}
By modifying, if necessary, $ \g_1 $ and $ \g_2 $, we may assume that
\begin{eqnarray}\label{}
\nn [\g_1,\g_1] &=&\g_2. 
\end{eqnarray}
By choosing exponential coordinates, we shall realize the corresponding connected, simply connected Lie group $G=\exp \g$  as 
\begin{eqnarray}\label{}
\nn G&=&\g\  (\text{as underlying manifold}), 
\end{eqnarray}
endowed with the  Campbell-Baker-Hausdorff product
$$
\nn x\cdot y:=x+y=\frac{1}{2}[x,y] ,\quad x,y\in G=\g.  
$$
Then $ 0$  is the identity element of  $G $ and $ x\inv=-x, \ x\in G,$
and for $ x_1,y_1\in\g_1,x_2,y_2\in\g_2 $ we have that
$$
 (x_1+y_1)\cdot (y_1+y_2)=(x_1+y_1)+(x_2+y_2+\frac{1}{2}[x_1,y_1] ). 
$$
For every  $ V\in \g $, we obtain a left-invariant vector field $V$ on $G$ by differentiating on the right in the direction of $ V $:
$$
 V f(x):=\frac{d}{dt}f(x\cdot (tV))|_{t=0},\quad f\in C^{\iy}(G), x\in G. 
$$
Furthermore, the group $ G $ admits a one-parameter group of automorphic dilations $\si_t, t\in\R^*,$ given by 
$$
 \si_t(x_1+x_2):= t x_1+t^2 x_2,\quad  x_1\in\g_1,x_2\in \g_2. 
$$

Let $ V_1,\cdots, V_m $ be a basis of $ \g_1 $ and $ U_1,\cdots, U_l $ be a basis of $ \g_2 $. We choose  the scalar product on $ \g $ for which these  vectors form an orthonormal basis, and denote by $|\cdot|$ the corresponding Euclidean norm on $G.$ Notice  that this basis then allows to identify $ \g_1\times \g_2 $ with the Euclidean space $ \R^{m}\times\R^{l}.$

Let $L_A$ be the element of the enveloping algebra $\U(\g)$ of $\g$ given by 
$$
 L_A:=\sum_{j,k=1}^{m}a_{jk}V_jV_k.
$$
Here we assume that the $ m\times m $ matrix $ A=(a_{j k}) $ is a real and symmetric. 
As usually, we shall consider the elements of  $\U(\g),$ such as $L_A,$ as  left-invariant differential operator on the group $G$. 

\begin{example}\label{sublapl}
\rm  If $A=\Id\in M_m(\R),$ then 
$$
L_{\Id}=\LL:=\sum_j V^{2}_j 
$$
is a {\it sub-Laplacian} on $ G,$ and for  suitable functions $ f $ on $ G $, the solution to the associated heat equation
$$
 \partial_t w=\LL w,\quad  t\in \R^*_+,  
$$
is given by 
$$
 w_t=f\ast h_t,  
$$
where  $h_t=e^{t\LL }, t>0,$ is  the family of  heat kernels associated to  $ \LL $.

Recall (see, e.g., \cite{VSC}) that the  {\it Carnot-Carath\'eodory (CC) distance} to the origin associated to the sum of squares operator $\LL$  is defined by
$$
\nn \no g_{cc}:=\inf_\ga |\ga|, \quad g\in G,
$$
where the infimum is taken over all absolutely continuous curves $ \ga: [0,T] \to G $ which are horizontal and connect $ 0\in G $ with $ g $, i.e,  $ \dot\ga(t)=\sum_{j=1}^{m}a_j(t)V_j(\ga(t)) $, for a.e. $t\in  [0,T].$ Here, $ |\ga| $  denotes the length of $\ga$ given by 
$$
\nn |\ga|:=\int_0^{T}(\sum_j a_j(t)^{2})^{{1/2}}dt. 
$$
Notice that the vector fields $ \{V_j\} $ are orthonormal with respect to the underlying sub-Riemanian geometry. 
The distance functions $ \|\cdot\|_{cc} $ is homogeneous of degree one under the family of dilations $\{\si_t\},$  as is the function
$$
 \no{ (x,u)}:=(\val x^{4}+16 \val u^{2})^{1/4}, 
$$
and it is well-known  \cite{VSC} that there exists a constant $ C\geq 1$ such that
\begin{eqnarray}\label{ccnoeq}
 C\inv \no g_{cc}\leq \no g\leq C \no g_{cc}, \quad \mbox{for all} \  g\in G. 
\end{eqnarray}
Moreover, since curves in $ \g_1$  are horizontal,
$\no{(x,0)}_{cc}=\val  x,$   and by relation (\ref{ccnoeq}), $\no{(0,u)}_{cc}\leq C \val u^{1/2},$
so that 
\begin{eqnarray}\label{}
\nn \no{(x,u)}_{cc}^{2}&\leq&(\no{(x,0)}_{cc}+\no{(0,u)}_{cc})^{2}
\leq(\val x+C \val u^{{1/2}})^{2} \\
\nn&\leq&(1+\ve)\val x^{2}+C_\ve \val u, 
\end{eqnarray}
for every $\ve>0.$
Furthermore, since the projection of a horizontal curve to $ \g_1 $ along $ \g_2 $ is horizontal too, it follows that
$$
 \no{(x,u)}_{cc}\geq\no{(x,0)}_{cc}=\val{x}, 
$$
and by (\ref{ccnoeq}) we have 
$
 \no{(x,u)}_{cc}\geq\frac{1}{C}\val{u}^{{1/2}}, 
$
hence
$$
 \no{(x,u)}^{2}_{cc}\geq(1-\ve)\val x^{2}+\frac{1}{C_\ve} \val u .
$$
So finally, for every $ \ve>0 $, there exists $ C_\ve\geq 1 $ such that
\begin{eqnarray}\label{encad}
 (1-\ve)\val x^{2}+\frac{1}{C_\ve} \val u\leq\no{(x,u)}_{cc}^{2}\leq   (1+\ve)\val x^{2}+C_\ve\val u,\quad  u\in \g_2,x\in\g_1.
\end{eqnarray}
It follows from \cite{VSC}, page 50, that we have the following point-wise estimate for the heat kernel $ h_t $ on  
$G$ (this is a special case of a more general result holding true for arbitrary nilpotent Lie groups):
$$
 h_t(g)\leq C_\ve t^{{-D/2}}e^{-\frac{\no g_{cc}^{2}}{4(1+\ve)t}},\quad  g\in G, t>0, \ve>0,  
$$
where $ D:=m+2l=\dim{\g_1}+2\dim{\g_2}$ denotes the homogeneous dimension of  $ G=\g_1\oplus\g_2.$ Consequently
\begin{eqnarray}\label{heatest1}
 h_t(x,u)\leq C_\ve t^{{-D/2}}e^{-\frac{\val x^{2}}{4(1+\ve)t}}e^{-\frac{\val u}{C_\ve t}},\quad  x\in\g_1,u\in \g_2, t>0, \ve>0.  
\end{eqnarray}
This estimate is essentially optimal, i.e., it is also known  \cite{VSC} that 
$$
h_t(g)\geq C_\ve t^{{-D/2}}e^{-\frac{\no g_{cc}^{2}}{4(1-\ve)t}}, \quad g\in G, t>0, 1>\ve>0,  
$$
 and so 
\begin{eqnarray}\label{heatest}
 h_t(x,u)\geq C_\ve t^{{-D/2}}e^{-\frac{\val x^{2}}{4(1-\ve)t}}e^{-C_\ve \frac{\val u}{t}},\quad  x\in\g_1,u\in \g_2, t>0, 0<\ve<1.  
\end{eqnarray}
 \end{example}
 
 \bigskip
 Let us now return to our operator 
 $$ L_A= \sum_{j,k=1}^{m}a_{jk}V_jV_k
$$
on $G,$ 
and consider  the initial value problem for the  associated Schr\"odinger equation
\begin{eqnarray}\label{schroL}
 \left\{\begin{array}{cc}
 \partial_t w(t,g)=i L_Aw(t,g),& t\in \R,g\in G, \\
 w(0,g)=f(g)\, , \ & 
\end{array}
\right .
\end{eqnarray}
with  $f\in L^{2}(G).$
If we assume that $ w\in C^1(\R, L^{2}(G)) $, then
\begin{eqnarray}\label{weitL}
\nn w(t,\cdot)&=&e^{it L_A}f, \quad t\in\R, 
\end{eqnarray}

In combination with the work by Ben-Sa\"id-Thangevalu, estimates (\ref{heatest1}) and  (\ref{heatest}) motivate our main 
\begin{theorem}\label{uniqueness}
Assume that there are constants $ C>0, \de>0, a>0,b>0, $ so that the solution $ w $ to the system (\ref{schroL}) for the operator $ L_A $  satisfies for some $ T\ne 0 $
\begin{eqnarray}\label{}
\label{woest} \val{w(0,(x,u))}&\leq& C e^{{-\frac{\val x^{2}}{4a}}}e^{-\de \val u}\\
\label{Loest}\val{w(T,(x,u))}&\leq&Ce^{{-\frac{\val x^{2}}{4b}}}e^{-\de \val u} , \quad x\in\g_1,u\in\g_2.
\end{eqnarray}
Then $ w\equiv 0 $ on $ \R\times G $ whenever 
\begin{equation}\label{avT}
ab < \noop A^{2}T^{2}.
\end{equation}
Here $ \noop A $ denotes the operator norm of the matrix $ A $ when acting on $ \R^{m}\simeq\g_1,$ with respect to the Euclidian norm $|\cdot|.$ 
 \end{theorem}
Notice that (\ref{encad}) shows that  the conditions (\ref{woest}) and (\ref{Loest}) in the theorem above are weaker than the corresponding conditions in Theorem \ref{ST}. 

\medskip

As usually, in the sequel we shall denote by $C$ a constant whose value may change from line to line.

\section{Proof of Theorem \ref{uniqueness}}

\setcounter{equation}{0}
\subsection{Preliminaries}

In order to prepare the proof  of Theorem \ref{uniqueness},  we first derive   the  following variant of Hardy's Theorem: 
\begin{corollary}\label{hrn} 
Let $ A=(a_{jk})\in M_n(\R) $ be a non-trivial real valued symmetric $ n\times n $ matrix. Moreover, 
let  $ F:\R^{n}\to \C $ be a measurable  function such that
\begin{eqnarray}\label{}
\nn \val {F(x)}&\leq& Ce^{
-\frac{\val {x^2}}{4a}
}
, \quad x\in\R^{n},\\
\nn \val{\widehat F(y)}&\leq& C
e^{
-\frac{
\val { A (y)}^{2}
}
{b}} , \quad y\in\R^{n},
\end{eqnarray}
for some positive constants $ a>0, b>0$. If $  ab<\noop A^{2}  $, then $ F=0 $.
\end{corollary}

\begin{proof}
Let $ \la $ be an eigenvalue of $ A $ of largest modulus. Then $ \noop A=\val\la $. Write $ \R^{n}=E_\la\oplus V $,
where $ E_\la $ denotes the eigenspace of $ \la $ and $ V $ the orthogonal complement of $ E_\la $ in $ \R^{n}. $
If $x\in E_\la,  v\in V $,  let us write $ F_x(v):=F(x+v) $. The function $ F_x $ is contained in $ L^{1}(V)$ for a.e.  $ x,$ and  if we put  $ F^{\nu}(x):=\int_VF(x+v)e^{-i \nu\cdot v}\,dv,\,  x\in E_\la,\nu\in V,$ then 
\begin{eqnarray}\label{}
\nn F^{\nu}(x)=\widehat{F_x}(\nu),\quad \nu\in V, x\in E_\la, 
\end{eqnarray}
so that 
\begin{eqnarray}\label{}
\nn \val{F^\nu(x) }&\leq& C \int_V e^{-\frac{ \val {x+v}^{2}}{4a}}dv 
= (C\int_V e^{-\frac{ \val {v}^{2}}{4a}}dv) \,e^{-\frac{ \val x^{2}}{4a}}
C_{a }\, e^{-\frac{ \val x^{2}}{4a}}, \quad x\in E_\la.
\end{eqnarray}
Furthermore we have that for $\xi\in E_\la$
\begin{eqnarray}\label{}
\nn \vert \widehat {F^\nu}(\xi)\vert&=&\vert \int_{E_\la}F^\nu(x)e^{-i \xi\cdot x}dx\vert
=\vert \int_{E_\la}\int_V F(x+v)\, e^{-i \nu\cdot v}\,dv\, e^{-i \xi\cdot x}\,dx\vert\\
\nn  &=&\vert \widehat F(\xi+\nu)\vert\le C e^{-\frac{\vert A(\xi+\nu)\vert^{2}}{b}}
=C e^{-\frac{\vert A(\xi)\vert^{2}+\vert A(\nu)\vert^{2}}{b}}
= C e^{-\frac{\vert A(\nu)\vert^{2}}{b}}  e^{-\frac{\vert A(\xi)\vert^{2}}{b}}\\
\nn  &=&C_{\nu,b} e^{-\frac{\val\la^{2} \val \xi^{2}}{b}}=C_{\nu,b} e^{-\frac{\noop A^{2} |\xi|^{2}}{b}}.
\end{eqnarray}
Applying  Hardy's  classical theorem \ref{Hardy} to the function $ F^v $, we see that $ F^v=0 $ for every $ v\in V $, since  $ a(b/\noop A^{2})  <1 $.
Therefore the function $V\ni  w\mapsto F(x+w) $ vanishes  for every $ x\in E_\la,$ and finally $ F=0 $. 
 \end{proof}
 
 Coming back to our group $G,$ let us  put as before $ L_A $ and 
\begin{eqnarray}\label{Ttdefined}
\nn T_t&:=&e^{it L_A},\quad  t\in \R. 
\end{eqnarray}
Then $ T_t $ is a unitary operator on $ L^{2}(G)$ for every  $t\in \R.$

\begin{definition}
\rm  
Following \cite{MRAnn},  define for $ \mu\in \g_2^* $ the skew symmetric form $ \om_\mu$ on $ \g_1 $ by 
\begin{eqnarray}\label{ommudef}
\om_\mu(V,W)&=&\mu([ V,W] ),
\end{eqnarray}
and the corresponding matrix
\begin{eqnarray}\label{}
\nn (J_\mu)_{jk}&:=&\mu([V_j,V_k] ) .
\end{eqnarray}
Given the real symmetric matrix $A\in M_m(\R),$ we also put
\begin{eqnarray}\label{smu}
S_\mu:=-AJ_\mu \in M_m(\R). 
\end{eqnarray}
 \end{definition}

Two possible scenarios may arrive (\cite{MRAnn}) :
\begin{itemize}
\item $ \om_\mu $ is non-degenerate for generic  $ \mu\in\g^*_2 $, (i.e., on a Zariski open subset). 
\smallskip

\noi In this case,  necessarily $ m=2n $ is even, and we  call $\g$  a {\it MW-algebra (``Moore-Wolf'')}.
\item  $ \om_\mu $ is degenerate for every $ \mu\in\g_2^* $, i.e., $ \g $ is not $ MW $.
 \end{itemize}
 
 \subsection{The case when $ \g $ is MW}

If $ \mu\in\g_2^* $ is such that $ \om_\mu $ is non-degenerate, then
\begin{eqnarray}\label{}
\nn (\om_\mu^{\wedge n})&=&Pf(\om_\mu)dx, 
\end{eqnarray}
where $ dx=dx_1\wedge\cdots\wedge dx_m  $ and $ Pf(\om_\mu) $ denotes the Pfaffian of $ \om_\mu $. In fact 
\begin{eqnarray}\label{}
\nn Pf(\om_\mu)&=&\pm \sqrt{\det J_\mu}. 
\end{eqnarray}
We put 
\begin{eqnarray}\label{detjmuedx}
 d_\mu x&:=&\val{Pf(\om_\mu)}dx=\sqrt{\det J_\mu}dx, 
\end{eqnarray}
regarded as a positive measure on $ \g_1 $.

The {\it $\mu $-twisted convolution} of two suitable functions or distributions on $ \g_1 $ is defined as
\begin{eqnarray*}
\nn \phi\times_\mu \psi(x)&:=&\int_{\g_1}\phi(x-y)\psi(y)e^{-i\pi \om_\mu(x,y)}d_\mu y\\
\nn &=&\int_{\g_1}\phi(y)\psi(x-y)e^{i\pi \om_\mu(x,y)}d_\mu y.
\end{eqnarray*}
Given $ f\in L^{1}(G) $ we put
\begin{eqnarray}\label{}
\nn f^{\mu}(x)&:=&\frac{1}{\val{Pf(\om_\mu)}}\int_{g_2}f(x,u)e^{-2\pi i \mu(u)}du.
\end{eqnarray}
Then for $ f, g\in L^{1}(G) $,
\begin{eqnarray}\label{twistcon}
(f\ast g)^{\mu}&=&f^\mu\times_\mu g^{\mu},
\end{eqnarray}
where $ \ast $ denotes convolution on the group $ G $. 
Fourier-inversion on $ \g_2\simeq\R^{l} $ yields
\begin{eqnarray}\label{fouinv}
f(x,u)&=&\int_{\g_2^*}f^{\mu}(x)e^{2\pi i \mu(x)}\val{Pf(\om_\mu)}d\mu. 
\end{eqnarray}

If $D\in\U(\g),$ we define the differential operator $D^\mu$ on $\g_1$ by the relation
$$
(Df)^\mu=D^\mu f^\mu, \qquad f\in\SS(G),
$$
where $\SS(G)$  denotes the space of Schwartz functions on $G.$  In particular, if $ X\in\g_1 $, then
\begin{eqnarray*}
 \nn( Xf)^{\mu}(x) &= &\int_{\g_2}\frac{d}{dt}f(x\cdot (tX)+u)|_{t=0}\, e^{-2\pi i \mu(u)} du=
 \partial_X f^{\mu}(x)+i\pi \om_\mu(x,X)f^{\mu}(x),
\end{eqnarray*}
so that
\begin{equation}\label{xmuexp}
X^\mu=\partial_X+i\pi \om_\mu(x,X).
\end{equation}
In a similar way, if  $ U\in\g_2 $, then
$$
U^\mu= 2i\pi \mu(U).
$$

 
 In particular, wee see that 
 \begin{equation}\label{lmudef}
(L_A f)^\mu=L_A^{\mu}f^{\mu}=\sum_{j,k}a_{jk} V^{\mu}_jV^{\mu}_k (f^{\mu}), \quad f\in \SS(G),
\end{equation}
where the $V_j^\mu$ are given by \eqref{xmuexp},
and we find that 
\begin{eqnarray}\label{}
\nn (T_tf)^{\mu}&=&e^{it L_A^{\mu}}(f^{\mu}),
\end{eqnarray}
Let us put
\begin{eqnarray}\label{ttmu}
 T_t^{\mu }\phi&:=&e^{i\frac{t}{4\pi}L_A^{\mu}}\phi, 
\end{eqnarray}
so that 
\begin{eqnarray}\label{ttmudef}
\nn (T_t f)^{\mu}&=&T^{\mu}_{4\pi t}(f^{\mu}). 
\end{eqnarray}
Since $ \om_\mu $ is non-degenerate, we can choose a symplectic basis $\V_\mu'=\V'=\{ V'_1,\cdots, V'_m\} $ of $ \g_1 $. 

Let $ R(\mu) $ be the transition matrix, i.e.,  $ V_j'=\sum_i R_{i,j}(\mu)V_i.  $ 

We can choose these bases $ \V_\mu' $ in an analytic way on an 
open  dense subset $ \U $ of the unit sphere in   $ \g^*_2 $ and then define for  $ \mu $ in the open  cone $\CC:= \R^*_+ \U $  of generic $\mu\ne 0$ the matrix $ R(\mu) $ by

\begin{eqnarray}\label{choose rmu}
 R(\mu)&=&{\vert\mu\vert ^{-1/2}}{R\Big(\frac{\mu}{\vert\mu\vert}\Big)}.
\end{eqnarray}
With respect to this basis,  we have that $ \om_\mu(V_j',V_k')=J_{jk} $, where
\begin{eqnarray*}
\nn J=(J_{ij})&=&
 \left(\begin{array}{cc}
0_n & I_n\\
-I_n & 0_n\\
\end{array}
 \right).
\end{eqnarray*}
In the new, symplectic coordinates  coordinates 
\begin{equation}\label{sympcoo}
 z=R(\mu)\inv x 
\end{equation} 
for $\g_1,$  our operator $ L_A^{\mu} $ is given by 
\begin{eqnarray}\label{newop}
\nn \tilde L_A^{\mu}&=&\sum_{j,k}a_{j k}(\mu)\tilde V'_j \tilde V'_k, 
\end{eqnarray}
where $ A(\mu)=R(\mu)\inv A \trans R(\mu)\inv $ and 
\begin{eqnarray}\label{expl}
\nn \tilde V'_j f(z)&=&\frac{\partial f}{\partial z_j}f(z)-i\pi (Jz)_j f(z).
\end{eqnarray}
Then the $ \tilde V'_j  $'s  are exactly the operators which arise from the usual vector fields on the Heisenberg group $ \mathbb H_n$ after applying the partial  Fourier transform in the central direction at $ \mu=1 $.

Since $ \om_\mu $ is represented by the matrix $ J $ in the new basis, we have that 
\begin{eqnarray}\label{}
\nn J&=&^{t}\! R(\mu)J_\mu  R(\mu) , 
\end{eqnarray}
so that by (\ref{detjmuedx})
\begin{eqnarray}\label{dzdmu}
 (\det J_\mu)\det (R(\mu)^{2}=1 \ \text{ and }d_\mu x =dz.
\end{eqnarray}

We can make use of   the symplectic coordinates $z$  in order to refer to  results from \cite{MRIn}. In particular, $ (T^{\mu}_t)_{t\in\R} $ is a well-defined one parameter group of unitary operators on $ L^{2}(\g_1,d_\mu) $. If $ \tilde f $ represents $ f $ in the symplectic coordinates, i.e.,
$$ \tilde f(z)=f(R(\mu)z),
$$ 
and if $ \tilde T^{\mu}_t $ represents $ T^{\mu}_t $ in these coordinates, then 
\begin{eqnarray}\label{exptilde}
 \tilde T^{\mu}_t \tilde f&=&\tilde f\times \tilde \ga_t^{\mu},  
\end{eqnarray}
where
\begin{eqnarray}\label{}
\nn \tilde f\times \tilde g(z)&=&\int_{\R^{m}}\tilde f(z-w)\tilde g(w)e^{-i \pi ^{t\! }w J z}\,dw , 
\end{eqnarray}
and  where the explicit formulas for $ \tilde\ga^{\mu}_t $ can be derived from \cite{MRIn}.
These formulas depend on the spectrum of the matrix 
\begin{eqnarray}\label{s(mu)}
S(\mu):=-R(\mu)\inv A J_\mu R(\mu)=-A(\mu)J, 
\end{eqnarray}
which is conjugate to the matrix
$$
\nn S_\mu:=-A J_\mu. 
$$
Actually $ S_\mu\in\s\p(\g_1,\om_\mu)  $, i.e., 
\begin{eqnarray}\label{}
\nn \om_\mu(V,S_\mu W)&=&-\om_\mu 
(S_\mu V,W), \quad V,W\in\g_1
\end{eqnarray}
(c.f. \cite{MRAnn} ).
\begin{remark}\label{lierel}
  The formula $ S=-AJ $ establishes a bijective relation between the elements $ S\in\s\p(n,\R) $ and the space of the real symmetric $ n\times n $ matrices $ A=(a_{jk}) $. If we define correspondingly 
\begin{eqnarray}\label{}
\nn \tilde \DE_S&:=&\sum_{j,k}a_{jk}\tilde V'_j\tilde V'_k,  
\end{eqnarray}
then 
\begin{eqnarray}\label{brack}
 [\frac{i}{4\pi}\tilde \DE_{S_1},\frac{i}{4\pi}\tilde \DE_{S_2}] &=&\frac{i}{4\pi}\tilde \DE_{[S_1,S_2]} .
\end{eqnarray}
This suggests that the operators $ \pm e^{{\frac{i}{4{\pi}}}\tilde\DE_S} $,  $ S\in \s\p(n,\R) $, generate a Lie group $ \tilde M(n,\R) $ with Lie algebra $ \s\p(n,\R) $. Indeed, this group is isomorphic to the metaplectic group, a two-fold cover of the symplectic group $Sp(n,\R),$ as has been shown by R. Howe in \cite{Howe}, where he realizes the metaplectic group as a boundary of the the called ``oscillator semi-group''. 
\smallskip

In particular, $ (\tilde T^{\mu}_t)_{t\in\R} $ is the one-parameter subgroup of $ \tilde M(n,\R) $ generated by $ \tilde \DE_{S(\mu)} $.
\end{remark}
 
In order to indicate the dependency of $ \tilde T^{\mu}_t $ on the matrix $ S(\mu)\in \s\p(n,\R) $, let us write 
\begin{eqnarray}\label{tildetsddef}
\tilde T_{t,S}&:=&e^{\frac{it}{4\pi}\tilde \DE_S}, \quad S\in \s\p(n,\R), t\in\R,  
\end{eqnarray}
so that
\begin{eqnarray}\label{}
\nn \tilde T^{\mu}_{4\pi T}&=&\tilde T_{4\pi T,S(\mu)}. 
\end{eqnarray}

   \bigskip
 Let us return to the proof of Theorem \ref{uniqueness}. 
We denote in the sequel the function $ w(0,\cdot ) $ of the theorem by $ f $. 
Following \cite{ST}, we begin with  the following simple 
\medskip

\noi {\bf Observation:} {\it It will suffice to show that (\ref{woest}) and (\ref{Loest}) imply that $ f^{\mu}=0 $ for all $\mu$ in the cone $\CC$ of generic $\mu$ which are  contained in a sufficiently small  neighbourhood of the origin  in $ \g_2^* \setminus \{0\}$. }

\smallskip 
Indeed,  condition (\ref{woest}) tells us that for a.e. 
 $x,$  the function $ \mu\mapsto f^{\mu}(x) $ is well defined and holomorphic for $ \mu\in (\g_2^*)_\C\simeq \C^{l},   \val{\Im\mu} <{\de}/{2\pi}.$ Hence by the identity principle for holomorphic functions,  our assumption implies that  $ f^{\mu}(x)=0 $ for all $ \mu\in \g_2^*,$ and therefore $ f=0 $. 
\medskip

Let us therefore assume in the sequel that $\mu\in\CC\subset \g_2^*.$
Then, by  (\ref{woest}) and (\ref{Loest}), 
\begin{eqnarray}\label{}
\label{woen} \val{f^{\mu}(x)}&\leq&C e^{-\frac{\val x^{2}}{4a}}\\
\label{Loen}  \val{T^{\mu}_{4\pi T}(f^{\mu})(x)}&\leq &C e^{-\frac{\val x^{2}}{4b}}, \quad x\in\g_1.
\end{eqnarray}
In the $ z $-coordinates, these estimates are equivalent to 
\begin{eqnarray}\label{}
\label{woenz} \val{\tilde f^{\mu}(z)}&\leq&C e^{-\frac{\val{R(\mu)z}^{2}}{4a}}\\
\label{Loenz}  \val{\tilde T^{\mu}_{4\pi T}(\tilde f^{\mu})(z)}&\leq &C e^{-\frac{\val {R(\mu)z}^{2}}{4b}},\quad  z\in\R^{m}.
\end{eqnarray}

As mentioned before, (c.f. \eqref{exptilde}), the operators $ \tilde T_{t,S} $ are twisted convolution operators of the form
\begin{eqnarray}\label{}
\nn \tilde T_{t,S}\va&=&\va\times \tilde \ga_{t,S},
\end{eqnarray}
where the $ \tilde \ga_{t,S} $ have been computed explicitly in \cite{MRIn}. Generically, the $ \tilde \ga_{t, S} $ are purely  imaginary Gaussians. The expression for  $ \tilde \ga_{t, S} $ becomes particularly simple for small times $ t $, when $ S $ is regular. If $ S $ is degenerate, $ \tilde \ga_{t, S} $ will rather look like a purely imaginary Gaussian in one set of variables and like a Dirac measure in another set of variables. We shall avoid the latter case by first assuming that the matrix $A$ is non-degenerate, and subsequently reduce  the case when $ A $ is degenerate to the non-degenerate case by means of a  small perturbation  trick.

\subsubsection{The case when $ \g $ is MW and $ A $ is non-degenerate}\label{generic}

When $ A $, hence also $ S_\mu $ and $ S(\mu), $ are non-degenerate, say for $\mu$ in the cone $\CC,$ then we can make use of the following results from (\cite{MRIn}, Prop. 2.3 and Theorem 3.1).
\begin{proposition}\label{tildega}
Let $ S\in\s\p(n,\R) $ be non-degenerate and choose $ \kappa $ so small, that 
$$ \det(\sinh(tS))\ne 0,\  0<\val t<\kappa .$$
Then for $ 0<\val t<\kappa,$
\begin{eqnarray}\label{gaexp}
 \nn \tilde \ga_{t,S}(z)=\frac{e^{-i\frac{\pi}{4}k(JS)\,\sgn(t)}}{2^{n}\val{\det(\sinh(tS/2))}^{\frac{1}{2}}}e^{-\frac{i\pi}{2} ^t\!zJ \coth(t S/2) z }. 
\end{eqnarray}
Here $ k(JS) $ denotes the number of positive eigenvalues of the symmetric matrix $ JS $ minus the number of its negative eigenvalues.
 \end{proposition}
 
 Passing back to our original coordinates $x,$ we obtain
 
 \begin{corollary}\label{s2.6}
 If $A$ is non-degenerate, then for $\phi\in\SS(\g_1),$
\begin{eqnarray*}
 T^\mu_t\phi &= &\phi\ti_\mu \ga_{t,S_\mu} ,
\end{eqnarray*}
where $S_\mu=-AJ_\mu$ and 
\begin{eqnarray}\label{gaexp}
 \nn  \ga_{t,\mu }(x)=\frac{e^{-i\frac{\pi}{4}k(J_\mu S_\mu)\, \sgn (t)}}{2^{n}\val{\det(\sinh(tS_\mu/2))}^{\frac{1}{2}}}e^{-\frac{i\pi}{2} ^t\!xJ_\mu \coth(t S_\mu/2) x }. 
\end{eqnarray}
\end{corollary}
\proof

Let us recall that
\begin{eqnarray}\label{}
\nonumber A(\mu)&=& R(\mu)\inv A \trans R(\mu)\inv \\
\nn J&=&{^{t}\!R(\mu)}J_\mu R(\mu),\\
\nn  S(\mu)&=& -A(\mu)J=- R(\mu)\inv A \trans R(\mu)\inv \trans R(\mu)J_\mu R(\mu)\\
\nn  &=&R(\mu)\inv S_\mu  R(\mu).
\end{eqnarray}
Therefore
\begin{eqnarray}\label{equal}
   k(J_\mu S_\mu)&=&k(JS(\mu)) ,\\
\nn \det(\sinh(tS(\mu)/2))&=&\det(\sinh(tS_\mu/2)),\\
\nn  J\coth(tS(\mu)/2)&=&\trans R(\mu)J_\mu\coth(tS_\mu/2)R(\mu),
\end{eqnarray}
where we have made use of  Sylvester's theorem in the first identity. 
Then  (compare \cite{MRMaAn} (1.14))
\begin{eqnarray}\label{}
\nn T^{\mu}_{t}\phi(x)&=&\tilde T^{\mu}_{t,S(\mu)}(\phi\circ R(\mu))( R(\mu)\inv(x))\\
\nn  &=&((\phi\circ R(\mu))\times \tilde \ga_{t,S(\mu}))( R(\mu)\inv(x)). 
\end{eqnarray}
Putting
\begin{eqnarray}\label{}
\nn c_{t,\mu}&:=&\frac{e^{-i\frac{\pi}{4}k(JS(\mu))\,\sgn(t)}}{2^{n}\val{\det(\sinh(tS(\mu)/2))}^{\frac{1}{2}}}=\frac{e^{-i\frac{\pi}{4}k(J_\mu S_\mu)\,\sgn(t)}}{2^{n}\val{\det(\sinh(tS_\mu/2))}^{\frac{1}{2}}},\\
\nn C(t,\mu)&:=&J\coth(tS(\mu)/2)=\trans R(\mu)J_\mu\coth(tS_\mu/2)R(\mu),\\
\nn e_{ C(t,\mu)}(z)&:=&e^{-i\frac{\pi}{2} z^tC(t,\mu) z },
\end{eqnarray}
we find that
\begin{eqnarray}\label{}
\nn T^{\mu}_t \phi(x)&=&c_{t,\mu}((\phi\circ R(\mu)) \times  e_{ C(t,\mu)}(R(\mu)\inv x), \quad x\in\g_1,
\end{eqnarray}
and so by \eqref{equal} and \eqref{dzdmu}
\begin{eqnarray}\label{explcom}
\nn &&
 T^{\mu}_{t}\phi(x)\\
\nn  &=&
c_{t,\mu}((
\phi\circ R(\mu)) \times e_{ C(t,\mu)}(R(\mu)\inv x)\\
\nn  &=&
c_{t,\mu}\int_{\g_1}\phi(R(\mu)y)\,e_{C(t,\mu)}(R(\mu)\inv(x)-y)
e^{i\pi \trans( R(\mu)\inv(x))Jy}dy\\
\nn  &=&
c_{t,\mu}\int_{\g_1}\phi(R(\mu)y)e^{-i\frac{\pi}{2}(^t\!(R(\mu)\inv(x)-y)C(t,\mu) ((R(\mu)\inv(x)-y) }
e^{i\pi\trans( R(\mu)\inv(x))Jy}\,dy\\
\nn  &=&
c_{t,\mu}\vert\det(R(\mu))\vert\inv\int_{\g_1}\phi(y)e^{-i\frac{\pi}{2} ^t\!(R(\mu)\inv(x-y))C(t,\mu) R(\mu)\inv(x-y) }
e^{i\pi(^{t}\! x ^{t}\!R(\mu)\inv JR(\mu)\inv y)}dy\\
\nn  &=&
c_{t,\mu}\int_{\g_1}\phi(y)e^{-i\frac{\pi}{2} ^t\!(x-y)J_\mu \coth(t S_\mu/2)  (x-y) }
e^{i\pi^{t}\! xJ_\mu  y}d_\mu y
,
\end{eqnarray}
which proves our claim.
\qed

\medskip

Since the matrix $ J_\mu\coth(tS_\mu/2)$ is  symmetric, it follows that for our function $f=w(\cdot,0)$ we have 
\begin{equation}\label{fourfor}
 T^{\mu}_{t}f^\mu(x)  =  c^{1}_{t,\mu}(x)\int_{\g_1}f_t^\mu( y)e^{i\pi \trans (J_\mu(\coth(tS_\mu/2)-\Id)x) y }d_\mu y,
\end{equation}
where 
$$c^{1}_{t,\mu}(x):= c_{t,\mu}e^{-i\frac{\pi}{2} \trans x J_\mu\coth(tS_\mu/2)x }, $$ 
and where
\begin{eqnarray*}
 \nn f_t^{\mu}(x) := e^{-i\frac{\pi}{2} \trans x J_\mu\coth(tS_\mu/2) x}\,f^\mu(x). 
\end{eqnarray*}
Define now the Euclidean Fourier transform on $\g_1$  as in Hardy's theorem by
\begin{eqnarray}\label{defhardy}
\nn \widehat \phi(x)&:=&\int_{g_1}\phi(y)e^{{-i x\cdot y}}\,dy, \quad x\in \g_1. 
\end{eqnarray}
Then by (\ref{fourfor})
\begin{eqnarray}\label{ttfhat}
 T^\mu_t f^\mu (x)&=&c^{1}_{t,\mu}(x)\widehat{f_t^\mu}(-\pi J_\mu(\coth(tS_\mu/2)-\Id)x ).
\end{eqnarray}

For $t=4\pi T \ne 0$ fixed and $\mu\in\CC$ sufficiently small, we have that $J_{\mu}( \coth(t S_{ \mu}/2)-\Id)$ is invertible. In fact,
\begin{eqnarray*}
 \coth(t S_{ \mu}/2)-\Id =  (\Id+O(|\mu|^2))\Big(\frac {tS_\mu}2\Big)^{-1}-\Id=\frac 2 t(\Id+O(|\mu|^2)-\frac t2S_\mu)S_\mu^{-1}
 =\frac 2 t(\Id+O(|\mu|))S_\mu^{-1},
\end{eqnarray*}
so that 
\begin{eqnarray}
\nn (J_{\mu}( \coth(t S_{ \mu}/2)-\Id))\inv  &=& \frac t2 S_\mu(\Id+O(|\mu|))J_\mu^{-1}
= \frac t2 S_\mu J_\mu^{-1}+O(\mu)\\
&=& -\frac t2 A+O(|\mu|). \label{estforA} 
\end{eqnarray}

Putting  
$$ v:=-\pi J_{\mu}( \coth(t S_{ \mu}/2)-\Id)(x), $$ 
in combination with \eqref{ttfhat} this implies that for $\mu\in\CC$ sufficiently small,
\begin{eqnarray}
 \nn \val {\widehat{f_t^\mu}(v)}
  &\leq&|c_{t,\mu}|^{-1}
 \vert{T^{\mu}_t f^{\mu}(-(\pi J_{\mu}( \coth(t S_{ \mu}/2)-\Id)\inv (v))}\vert\\
 &=&|c_{t,\mu}|^{-1}\vert{T^{\mu}_t f^{\mu}((2TA+O(|\mu|)) v)}\vert \label{estima}.
 \end{eqnarray}

Combining (\ref{woen}), (\ref{Loen}) and  (\ref{estima}), we see that for $x,v\in \g_1$
\begin{eqnarray}\label{}
\label{woenn} |f_t^\mu(x)|&=&|f^\mu(x)| \leq  e^{-\frac{\val x^{2}}{4a}}\\
\label{Loenn2}  \val{\widehat {f_t^\mu}(v)}&\leq & C|c_{t,\mu}|^{-1}e^{-\frac{|(2TA+O(|\mu|))v |^2}{4b}}.
\end{eqnarray}
Since $A$ is non-degenerate, given $\ve>0$ we may choose  some  $\de>0 $, such that for every $ v\in\g_1$ and $\mu\in\CC$ with $ \val \mu<\de $,
$$
|(2TA+O(|\mu|)) (v)|\ge \frac {|2TA v|}{1+\ve},
$$
so that 
\begin{equation}\label{Loenn2} 
\val{\widehat {f_t^\mu}(v)}\leq  C|c_{t,\mu}|^{-1}e^{-\frac{|TA v|}{b(1+\ve)^2}}.
\end{equation}

Assume now  that  $ ab<T^{2}\noop A^{2}.$ Then we may choose $\ve>0$ so small that $ab(1+\ve)^2<T^{2}\noop A^{2}.$ Then the relations  (\ref{woenn}) and (\ref{Loenn2}) in combination with Corollary \ref{hrn} imply that for $\mu\in\CC$ sufficiently small we have that $f_t^\mu= 0,$ hence also $f^\mu=0.$ By our previous Observation, this concludes the proof of Theorem \ref{uniqueness} in this sub-case.

\subsubsection{The case when $ \g $ is MW and $ A $ is degenerate}\label{regdegen} $ $
Let us next assume that $ \g $ is MW, but that $ A\ne 0 $ is singular. We shall reduce this case to the previous one   
by means of a perturbation argument. 

To this end, we shall have to perform some calculations within Howe's oscillator semigroup, which is why we shall work within the symplectic coordinates $z$ for $\g_1$ given by \eqref{sympcoo}.
\medskip

Recall first that the sub-Laplacian $\LL= \sum_j V_j^{2} $, which generates the heat semi-group  with heat kernels $ h_t $ on $ G $, corresponds to the matrix $A= \Id $. We therefore put for $ \mu\in \CC $
\begin{eqnarray}\label{dmubudef}
B(\mu)&:=&-R(\mu)\inv \Id\, ^{t}\!R(\mu)\inv=-{R(\mu)\inv} ^{t}\!R(\mu)\inv,\\
\nn D(\mu)&:=&-B(\mu)J.
\end{eqnarray}
Since $ -B(\mu) $ is positive definite, the operator  
\begin{eqnarray}\label{}
\nn \tilde P^{\mu}_\ta&:=&e^{\frac{i\tau}{4\pi}\tilde \DE _{D(\mu)}}, \quad \Im \tau\geq 0 , 
\end{eqnarray}
is a well-defined contraction operator on $ L^{2}(\g_1) $ when $ \Im \ta>0 $, lying in Howe' s oscillator semi-group (c.f. \cite{Howe}). Moreover, by Theorem 5.2 in \cite{MRAnMa}, 
\begin{eqnarray}\label{tildep}
\tilde P^{\mu}_\ta \phi&=&\phi \times \tilde\Ga^{\mu}_\ta, \quad \Im \ta\geq 0,  
\end{eqnarray}
where $ \tilde\Ga^{\mu}_\ta $ is given by
\begin{eqnarray}\label{exppa}
 \tilde\Ga^{\mu}_\ta(z)&=&\frac{1}{2^{n}{(\det(\sinh(\ta \frac{D(\mu)}{2})))}^{\frac{1}{2}}}e^{-\frac{i\pi}{2} ^t\!zJ \coth(\ta \frac{D(\mu)}{2}) z }, 
\end{eqnarray}
for a suitable choice of the square root. 

Observe that by  narrowing down the cone $\CC,$ we may assume that  for $\mu\in\CC$ the $ R(\mu) $'s are chosen  in (\ref{choose rmu}) so that
\begin{eqnarray}\label{rudom}
 \noop{R(\mu)\inv}&\leq&C \val \mu^{1/2},\quad \noop{D(\mu)}\leq C\val\mu, \quad  \noop{D(\mu)\inv}\leq C\val\mu^{-1}.
\end{eqnarray}

\begin{lemma}\label{gaest}
Given $ T\in \R\setminus\{0\}$ and $\ve>0,$  there exists an $ \ve_0=\ve_0(T,\ve)>0, \, \ve_0<\ve,$ such that for every $ \ve' \in ]0,\ve_0[,$ every $\mu\in\CC$ such that $|\mu|<\ve_0$  and every $\si\in[0,1]$ we have
\begin{eqnarray}\label{}
\nn \val{\tilde\Ga^{\mu}_{4\pi T \ve' (\si+i)}(z)}&\leq&C (\ve'  \val\mu)^{-n}e^{-\frac{\val{R(\mu)z}^{2}}{4\ve}}, \quad z\in\g_1. 
\end{eqnarray}

 \end{lemma}
\begin{proof}
If $ \ta=4\pi T \ve' (\si+i) $ then under our hypotheses, (\ref{rudom}) implies that 
\begin{eqnarray}\label{}
\nn J\coth(\frac{\ta D(\mu)}{2})&=&\frac{2}{\ta}JD(\mu)\inv +O(\val\ta |\mu|)\\
\nn &=&\frac{2}{\ta}\ ^{t}\!R(\mu)R(\mu)+O(\val\ta |\mu|),
\end{eqnarray}
so that 
\begin{eqnarray}\label{}
\nn \Re [\frac{-i\pi}{2}J\coth(\frac{\ta D(\mu)}{2})] &=&-\frac{^{t}\!R(\mu)R(\mu)}{4\ve' T (1+\si^{2})}+O(\ve' ). 
\end{eqnarray}
Hence, using again \eqref{rudom}, we have 
\begin{eqnarray}\label{}
\nn \Re [\frac{-i\pi}{2}\trans z J\coth(\frac{\ta D(\mu)}{2})z] &=&-\frac{\vert R(\mu)z\vert ^{2}}{4\ve' T (1+\si^{2})}+O(\ve' \val z^{2})\\
\nn&\le& -\frac{\vert R(\mu)z\vert ^{2}}{4\ve' T (1+\si^{2})}+C\ve'|\mu| |R(\mu)z|^2. 
\end{eqnarray}
Similarly, $$ \val{\det\sinh(\frac{\ta D(\mu)}{2})}^{{-1/2}}\leq C(\ve' \vert  \mu\vert )^{{-m/2}} .$$
The claim follows now easily.
 \end{proof}
Let  $ \ve>0,$ and assume that $ \ve'>0 $ and $\mu\in\CC$ are chosen sufficiently small so that the conclusion of  Lemma \ref{gaest} holds true. Since 
$ \va \times \psi(z)\leq \val\va\ast\val \psi(z) $, where $ \ast $ denotes ordinary convolution on Euclidean space, (\ref{Loenz}) and Lemma \ref{gaest} in combination with standard convolution identities for heat kernels on Euclidean 
space imply that
\begin{eqnarray}\label{tildePestim}
 \val{\tilde P^{\mu}_{4\pi T\ve' (1+i)}\tilde T^{\mu}_{4\pi T}\tilde{f^{\mu}}(z)}&\leq &C(\ve'\val \mu)^{-n} e^{-\frac{\val{ R(\mu)z}^{2}}{4(b+\ve)}}. 
\end{eqnarray}

Putting again $t:=4\pi T $, we may re-write 
\begin{eqnarray}\label{}
\nn \tilde P^{\mu}_{t\ve' (1+i)}\tilde T^{\mu}_t&=&e^{\frac{i t\ve'(1+i) }{4\pi}\tilde\DE_{D(\mu)}}e^{{\frac{it}{4\pi}}\tilde \DE_{S(\mu)}},\\
\nn &=&e^{\frac{- t\ve' }{4\pi}\tilde\DE_{D(\mu)}}e^{\frac{i t\ve' }{4\pi}\tilde\DE_{D(\mu)}}e^{{\frac{it}{4\pi}}\tilde \DE_{S(\mu)}}.
\end{eqnarray}
In view of Remark \ref{lierel}, we have 
\begin{eqnarray}\label{}
\nn e^{\frac{i t\ve' }{4\pi}\tilde\DE_{D(\mu)}}e^{{\frac{it}{4\pi}}\tilde \DE_{S(\mu)}}&=&e^{{\frac{it}{4\pi}}\tilde \DE_{S_{\ve' }(\mu)}},  
\end{eqnarray}
where 
\begin{eqnarray}\label{}
\nn S_{\ve' }(\mu)&=&\ve'  D(\mu)+S(\mu)+\frac{1}{2}\ve' [ D(\mu),S(\mu)] +\cdots 
\end{eqnarray}
is given by the Baker-Campbell-Hausdorff formula. But, similarly to (\ref{rudom}),
$$
\nn \noop{S(\mu)}\leq C\val\mu \quad \text{  on }\CC, 
$$
so that 
\begin{eqnarray}\label{estimsve}
 S_{\ve' }(\mu)=S(\mu)+\ve'  D(\mu)+O(\ve' \val\mu^{2}). 
\end{eqnarray}
Finally we may write 
\begin{eqnarray}\label{}
\nn \tilde P^{\mu}_{t\ve' (1+i)}\tilde T^{\mu}_t&=&e^{{\frac{it}{4\pi}}\tilde \DE_{S_{\ve' }(\mu)}} [e^{-{\frac{it}{4\pi}}\tilde \DE_{S_{\ve' }(\mu)}}e^{\frac{i \ve'(it) }{4\pi}\tilde\DE_{D(\mu)}}e^{{\frac{it}{4\pi}}\tilde \DE_{S_{\ve' }(\mu)}}].  
\end{eqnarray}
In view of (\ref{brack}), the second factor can be re-written as $ e^{{\frac{i\ve'(it) }{4\pi}}\tilde \DE_{D_{\ve' }(\mu)}} $, with
\begin{eqnarray}\label{Dvepdef}
 \ve' (it)D_{\ve' (\mu)}&=&\exp{{\rm ad}(-t S_{\ve' }(\mu))}(\ve' (it)D(\mu))\\
\nn &=&\ve' (it)D(\mu)-\ve' it^{2}[S_{\ve' }(\mu),D(\mu)] +\cdots\\
\nn &=&i\ve' t[D(\mu)+O(\val\mu^{2})], 
\end{eqnarray}
i.e.,
\begin{eqnarray}
 \nn D_{\ve' }(\mu)&=&D(\mu)+O(\val\mu^{2})\\
  &=&(\Id+O(\val\mu)) D(\mu), \label{dmu}
\end{eqnarray}
where we have applied \eqref{rudom}. So we have finally re-written 
\begin{eqnarray}\label{prodee}
 \tilde P^{\mu}_{t\ve' (1+i)}\tilde T^{\mu}_t&=&e^{{\frac{it}{4\pi}}\tilde \DE_{S_{\ve' }(\mu)}} e^{\frac{i \ve'(it) }{4\pi}\tilde\DE_{D_{\ve' }(\mu)}}.
\end{eqnarray}
Now, since the leading term of $ D_{\ve' }(\mu) $ is $ D(\mu) $, we obtain the following analogue of Lemma \ref{gaest}:
\begin{lemma}\label{gaest1}
Let $ e^{\frac{i \ve'(it) }{4\pi}\tilde\DE_{D_{\ve' }(\mu)}}\va=\va \times \tilde\GA ^{{\ve',\mu}}_{it}  $. Then the conclusion of Lemma \ref{gaest} holds for $ \GA ^{{\ve' ,\mu} }_{4\pi T i\ve'} $ in place of $ \GA^{\mu}_{4\pi T\ve' (\si+i)} $, i.e., 
given $ T\in \R\setminus\{0\}$ and $\ve>0,$  there exists an $ \ve_0=\ve_0(T,\ve)>0, \ve_0<\ve,$  such that for every $ \ve' \in ]0,\ve_0[$ and every $\mu\in\CC$ such that $|\mu|<\ve_0$  we have
\begin{eqnarray}\label{}
\nn \val{\tilde\Ga^{\ve' ,\mu}_{4\pi T i\ve' }(z)}&\leq&C (\ve'  \val\mu)^{-n}e^{-\frac{\val{R(\mu)z}^{2}}{4\ve}}, \quad z\in\g_1. 
\end{eqnarray}
 \end{lemma}
\begin{proof}
The function $ \tilde\Ga^{{\ve' ,\mu}}_{it} $ is again given by formula (\ref{exppa}), only with $ D(\mu) $ replaced by $ D_{\ve' }(\mu) $. Since by \eqref{rudom} and \eqref{dmu}, for $\mu\in\CC$ sufficiently small,
\begin{eqnarray}\label{}
\nn D_{\ve' }(\mu)\inv&=&D(\mu)\inv(I+O(\val\mu)),  
\end{eqnarray}
we obtain for $ \ta=i\ve' t=i \ve'4\pi T $ that
\begin{eqnarray}\label{}
\nn J\coth(\ta \frac {D_{\ve' }(\mu)}2)&=&\frac{2}{\ta}J  D(\mu)\inv (\Id+O(|\mu|)),  
\end{eqnarray}
hence 
\begin{eqnarray}\label{}
\nn \Re[-i\frac{\pi}{2}J\coth(\frac{i\ve't D_{\ve' (\mu)}}{2})] &=&-\frac{^{t}\!R(\mu)R(\mu)}{4\ve'  T}+O\Big(\frac1{\ve'}\Big). 
\end{eqnarray}
In combination with \eqref{rudom} this implies that 
\begin{eqnarray}\label{}
\nn \Re[-i\frac{\pi}{2}\trans z J\coth(\frac{i\ve't D_{\ve' (\mu)}}{2})z] &\le &-\frac{|R(\mu)z|^2}{4\ve'  T}+C_T|\mu|\frac{|R(\mu)z|^2}{4\ve'  T}. 
\end{eqnarray}
By choosing $\ve_0,$ hence $|\mu|,$ sufficiently small, we can now conclude in a similar way as in Lemma \ref{gaest}.
 \end{proof}
 For $\ve'>0$ as in Lemma \ref{gaest1}, let us put
 \begin{eqnarray}\label{}
\nn \tilde f^{\mu}_{\ve' }&=&e^{{-\frac{\ve'  4\pi T}{4\pi}}\tilde \DE_{D_{\ve' } (\mu)}}(\tilde f^{\mu}). 
\end{eqnarray}

Our assumption   (\ref{Loenz}) in   combination with Lemma \ref{gaest1} now imply that 
\begin{eqnarray}\label{ftildeve}
\val{\tilde f^{\mu}_{\ve' }(z)}&\leq&C{(\ve'\val \mu)^{-n} } e^{-\frac{\val{R(\mu)z}^{2}}{4(a+\ve)}}. 
\end{eqnarray}
Moreover, by (\ref{prodee}), we may re-write (\ref{tildePestim})  as
\begin{eqnarray}\label{newestb}
 \val{e^{iT\tilde \DE_{S_{\ve' }(\mu)}}(\tilde f^{\mu}_\ve)(z)}&\leq& C{(\ve'\val \mu)^{-n} } e^{-\frac{\val{R(\mu)z}^{2}}{4(b+\ve)}},  
\end{eqnarray}
where by \eqref{estimsve}
$$
 S_{\ve' }(\mu)=-A_{\ve'} (\mu)J,
$$
with
\begin{eqnarray}\label{}
\nn A_{\ve'} (\mu)&=&(S(\mu)+\ve'  D(\mu)+O(\ve' \val\mu^{2}))J\\
\nn &=&A(\mu)+\ve'  B(\mu)+O(\ve'  \val\mu^{2}).  
\end{eqnarray}
Now observe that in the original coordinates $x,$  $ A_{\ve' }(\mu) $ corresponds to the matrix
\begin{eqnarray}\label{}
\nn A^{\mu}_{\ve' }&=&A-\ve'  \Id +O(\ve' \val \mu), 
\end{eqnarray}
hence is regular for $ \ve'>0 $ and $ \mu $ sufficiently small. Notice here also that
$ \noop{A^{\mu}_{\ve' }}\to \noop A $ as $ \ve'\to 0 $. So we have essentially reduced ourselves to the case treated in Subsection \ref{generic}.

More precisely, for  $ \mu\in \CC $ sufficiently small, by passing back to the original coordinates $x,$ we can conclude as in Subsection \ref{generic} from (\ref{ftildeve}) and (\ref{newestb}) that 
\begin{eqnarray}\label{0smallve}
  \tilde f^{\mu}_{\ve' }=0 \text{ whenever  }ab< T^{2}\noop A,
\end{eqnarray}
 provided $ \ve$ (hence $\ve'<\ve$) is sufficiently small (the fact that $ A^{\mu}_{\ve' }$ now also depends on $\mu$  in a mild way has no consequences for our argument). Notice also that, according to Lemma \ref{gaest} and \ref{gaest1},  this conclusion holds true for all $\mu\in\CC$ such that $|\mu|<\ve_0(T,\ve),$  independently of $\ve',$ provided $\ve'$ is sufficiently small.

Finally if we put $ B_{\ve' }(\mu):=D_{\ve' }(\mu)J$, then by (\ref{Dvepdef}) $B_{\ve'}=B(\mu)+O(\val{\mu}^{2}), $ and  we see that $ \Re (-B_{\ve' }(\mu))>0 $ for $ \mu $ sufficiently small. Therefore Theorem 5.5 in \cite{MD}  implies that $ \tilde f^{\mu}_{\ve' }\to \tilde f^{\mu} $ in $ L^{2} (\g_1)$, since $\ve'T D_{\ve'(\mu)}\to 0$ as $ \ve' \to 0 $. So (\ref{0smallve}) forces  again  $ \tilde  f^{\mu}$ to be 0 for every $ 0\ne \mu \in \CC $ sufficiently small, and consequently  $ f=0 $. This concludes the proof of Theorem \ref{uniqueness} in the case considered in  Subsection \ref{regdegen}.
\subsection{The case when $ \g $ is not MW}\label{notmw}$  $

Assume finally  that $ \g $ is not MW. Following \cite{MRAnn}, we then define a new Lie algebra 
\begin{eqnarray}\label{}
\nn \h:=(\g_1\times \g_1^*)\oplus (\g_2\times\R)=\h_1\oplus\h_2,  
\end{eqnarray}
with the Lie bracket $ \h_1\oplus\h_2\to \h_2 $ given by
$$
[(V,\xi),(V' ,\xi' )] :=([V,V' ] ,\xi' (V)-\xi(V' )). 
$$
The Lie algebra $ \h $ is again nilpotent of step 2. One easily sees (c.f. \cite{MRAnn}) that if $ l=(\mu,\la)\in \h_2^*=\g_2^*\times\R $, with $ \la\ne 0  $, then $\om_l $ is non-degenerate on $ \h_1  $, so that $ \h $ and the corresponding group $ H=\exp\h $ is MW.

As for $ G $, we use exponential coordinates for $ H $, so that $ H=\h $ as the underlying manifold,  endowed with the Baker-Campbell-Hausdorff product. 

We embed $ G $ into $ H $ by means of the mapping $ (x,u)\to ((x,0),(u,0)) $ and define the closed abelian subgroup $ M $ of $ H $ by
\begin{eqnarray}\label{}
\nn M&=\{((0,\xi),(0,s)): \xi\in \g_1^*,s\in\R\}.
\end{eqnarray}
Then
$$
 H=M\cdot G, 
$$
and the mapping $ (m,g)\mapsto mg $ is a diffeomorphism from $ M\times G $ onto $ H $. If $ V $ is  a left-invariant 
vector field on $ G $, we denote by $ \tilde V $ its lift to $ H $, i.e.,
\begin{eqnarray}\label{}
\nn \tilde V F(h)&=&\frac{d}{dt}F(h\,\exp{ tV})|_{t=0},\quad  h\in H. 
\end{eqnarray}
If $ F:H\to \C $ is a function on $ H $ and if $ m\in M $, we define $ F_m:H\to \C $ by
$$
 F_m(h):=F(mh),\quad   h\in H.  
$$
Then clearly for any smooth function $ F:H\to \C $
\begin{eqnarray}\label{}
\nn \tilde V F(mg)&=&V F_m(g) \quad \mbox{for all} \  (m,g)\in M\times G. 
\end{eqnarray}
Let now 
$$
 L_A=\sum_{j,k}^{m} a_{jk}V_jV_k 
$$
on $ G $ be given as before. Then its lift $ \tilde L_A=\sum_{j,k}^{m} a_{jk}\tilde V_j\tilde V_k $ on $ H $ satisfies
\begin{eqnarray}\label{liftonH}
\nn (\tilde L_AF)(mg)&=&(L_AF_m)(g)\quad \mbox{for all} \  (m,g)\in M\times G. 
\end{eqnarray}
Assume now that the function $ f$ on  $  G $ satisfies the estimates (\ref{woest}),(\ref{Loest}) in Theorem \ref{uniqueness}, i.e.,
\begin{eqnarray}\label{}
\label{woesttt} \val {f(x,u)}&\leq& C e^{{-\frac{\val x^{2}}{4a}}}e^{-\de \val u}\\
\label{Loesttt}\val{(e^{i T L_A}f)(x,u)}&\leq&Ce^{{-\frac{\val x^{2}}{4b}}}e^{-\de \val u} , x\in\g_1,u\in\g_2.
\end{eqnarray}
We can define the lift $ \tilde f $ of $ f $ on $ H=MG $ by putting 
\begin{eqnarray}\label{deftif}
 \tilde f((\xi;s)\cdot (x,u))&=&e^{-\frac{\val\xi^{2}}{4a}}e^{-\de\val s}f(x,u). 
\end{eqnarray}
Here $ (\xi;s) $ is a short-hand writing for $ ((0,\xi), (0,s))\in M $ and $ (x,u) $ stands for $ ((x,0),(u,0)) $.

Moreover we shall identify $ \g_1\times \g_2 $ with $ \R^{m}\times\R^{l} $ by means of the mapping
\begin{eqnarray}\label{}
\nn (x,u)&\mapsto& \sum_{j=1}^{m}x_j V_j+\sum_{s=1}^{l}u_s U_s 
\end{eqnarray}
(compare Section \ref{intro}), and similarly we shall identify $ \g_1^* $ with $ \R^{m} $ by means of the dual basis to $ V_1,\cdots, V_m $. This allows us to assume that $ H=(\R^{m}\times \R^{m})\oplus(\R^{l}\times\R),$ and  that $ \val{\cdot } $ is  the Euclidean norm in (\ref{deftif}). 

Observe that 
\begin{eqnarray}\label{}
\nn (\xi;s)\cdot (x,u)&=&(x,\xi,u, s-\frac{1}{2}\xi\cdot x),
\end{eqnarray}
so that in exponential coordinates for $ H $
\begin{eqnarray}\label{tildefexp}
\nn \tilde f(x,\xi,u,s)&=&e^{-\frac{\val\xi^{2}}{4a}}e^{-\de\val {s+\frac{1}{2}\xi\cdot x}}f(x,u). 
\end{eqnarray}
Since
$$
\val{s+\frac{1}{2}\xi\cdot x}\geq \val s-\frac{1}{4}(\val\xi^{2}+\val x^{2}), 
$$
the relation (\ref{woesttt}) implies that 
\begin{eqnarray}\label{newest}
\nn \val{\tilde f(x,\xi,u,s)}&\leq&C e^{-\frac{\val x^{2}+\val \xi^{2}}{4a' }}e^{-\de\val{(u,s)}},  
\end{eqnarray}
with $ a' :=\frac{a}{1-\de a} $ if $ \de >0$ is small. Similarly (\ref{Loesttt}) implies that
\begin{eqnarray}\label{}
\nn \val{e^{i T \tilde L_A}\tilde f((\xi,s))\cdot(x,u)}&=&\val{(e^{it L_A} \tilde f_{(\xi,s)})(x,u)} \\
\nn &=&e^{-\frac{\val\xi^{2}}{4a}}e^{-\de\val s}\val{(e^{it L_A }f)(x,u)}\\
\nn &\leq&C e^{-\frac{\val\xi^{2}}{4a}}e^{-\frac{\val x^{2}}{4b}}e^{-\de \vert (u,s)\vert}.
\end{eqnarray}

{\bf Case 1: $ a\leq b $.} In this case 
\begin{eqnarray}\label{}
\nn \val{e^{i T \tilde L_A}\tilde f((\xi,s))\cdot(x,u)}&\leq&C e^{-\frac{\val\xi^{2}+\val \xi^{2}}{4b' }}e^{-\de \val{ (u,s)}},
\end{eqnarray}
with $ b' :=\frac{ b}{1-\de b}>b $, if $ \de $  is small enough. Now if $ ab<\noop A^{2 }T^{2} $, then we can choose $ \de  $ so small in (\ref{woest}), (\ref{Loest}) that $ a' b' <\noop A^{2}T^{2} $. Since $ H $ is MW, by what   has already been proved, we conclude that $\tilde f=0  $ on $ H $ and hence $ f=0 $ on $ G $. 

\medskip
{\bf Case 2: $ b<a $.}
This case can be reduced to the preceding one $ a<b $. Just replace $ f $ by $ g:=e^{iT L_A}f $ and $ T $ by $ -T $, so that $ e^{-iT L_A}g=f $. Apply the previous case to $ g $, which interchanges the roles of $ a $ and $ b $. Hence $ g=0 $ and finally then $ f=0,$ too.

\medskip
This concludes the proof of Theorem \ref{uniqueness}.

\medskip

Jean Ludwig, {\it Laboratoire de Math\'ematiques et Applications de Metz UMR 7122, Universit\'e de Lorraine, \^{i}le du Saulcy, 57045 Metz CEDEX
01, France. E-mail: jean.ludwig@univ-lorraine.fr}\\

Detlef M\"uller, {\it Mathematisches Seminar, C.A.-Universit\"at Kiel,
Ludewig-Meyn-Str.4, D-24098 Kiel, Germany. E-mail: mueller@math.uni-kiel.de }.


\begin{thebibliography}{9999}
\bibitem{BDJ} A. Bonami, B. Demange, P. Jaming. Hermite functions and uncertainty principles for
the Fourier and the windowed Fourier transforms, Rev. Mat. Iberoamericana 19,1 (2006)
23--55.
\bibitem{CF1} M. Cowling, J. F. Price. Generalizations of Heisenberg's inequality. Harmonic Analysis
(Cortona, 1982) Lecture Notes in Math. 992 (1983) 443--449, Springer, Berlin.
\bibitem{CEKKPV} M. Cowling, L. Escauriaza, C.E. Kenig, G. Ponce, L. Vega. The Hardy uncertainty principle revisited, preprint 2010
\bibitem{EKPV3} L. Escauriaza, C.E. Kenig, G. Ponce, L. Vega. The sharp Hardy uncertainty principle for Schr\"odinger evolutions. Duke Math. J. 155, 1 (2010) 163--187.
\bibitem{EKPV2} L. Escauriaza, C.E. Kenig, G. Ponce, L. Vega. Hardy's uncertainty principle, convexity
and Schr\"odinger evolutions. J. Eur. Math. Soc. 10, 4 (2008) 883--907.
\bibitem{EKPV1} L. Escauriaza, C.E. Kenig, G. Ponce, L. Vega. On uniqueness properties of solutions of Schr\"odinger equations. Comm. PDE. 31, 12 (2006) 1811--1823.
\bibitem{Hardy} G. H. Hardy, A theorem concerning Fourier transform. Journal London Math. Soc. 8,
(1933) 227--231.
\bibitem{Howe} R. Howe.  The oscillator semigroup. The mathematical heritage of Hermann Weyl (Durham, NC, 1987), 61--132, Proc. Sympos. Pure Math., 48, Amer. Math. Soc., Providence, RI, 1988. 

\bibitem{MRIn} D. M\"uller, F.  Ricci. Analysis of second order differential operators on Heisenberg groups. I. Invent. Math. 101 (1990), no. 3, 545--582. 
\bibitem{MRAnn} D. M\"uller, F. Ricci.  Solvability for a class of doubly characteristic differential operators on $2$-step nilpotent groups. Ann. of Math. (2) 143 (1996), no. 1, 1--49. 
\bibitem{MRMaAn} D. M\"uller, F. Ricci. Solvability for a class of non-homogeneous differential operators on two-step nilpotent groups. Math. Ann. 304 (1996), no. 3, 517--547. 
\bibitem{MRAnMa} D. M\"uller, F. Ricci.  Solvability of second-order left-invariant differential operators on the Heisenberg group satisfying a cone condition. J. Anal. Math. 89 (2003), 169--197.

\bibitem{MD} D. M\"uller. Local solvability of linear differential operators with double characteristics. II. Sufficient conditions for left-invariant differential operators of order two on the Heisenberg group. J. Reine Angew. Math. 607 (2007), 1--46. 

\bibitem{ST} S. Ben Sa\"{i}d, S.  Thangavelu. Uniqueness of solutions to the Schr\"odinger equation on the Heisenberg group. Preprint 2011.
\bibitem{VSC} N. Th. Varopoulos, L.  Saloff-Coste, T.  Coulhon.
Analysis and geometry on groups. 
Cambridge Tracts in Mathematics, 100. Cambridge University Press, Cambridge, 1992. 
\end{thebibliography}
 \end{document}